\documentclass[12 pt,reqno]{amsart}
\usepackage{amsfonts}
\usepackage{amscd,amssymb}

\theoremstyle{plain}
\newtheorem{theorem}[subsection]{Theorem}
\newtheorem{lemma}[subsection]{Lemma}
\newtheorem{proposition}[subsection]{Proposition}
\newtheorem{corollary}[subsection]{Corollary}
\theoremstyle{definition}
\newtheorem{remark}[subsection]{Remark}
\newtheorem{definition}[subsection]{Definition}
\newtheorem{example}[subsection]{Example}

\newtheorem{conjecture}[subsection]{Conjecture}
\numberwithin{equation}{section} 
\input{tcilatex}

\begin{document}
\title[Recurrence Relation for Jones Polynomials]{Recurrence Relation for
Jones Polynomials }
\thanks{This research is partially supported by Higher Education Commission,
Pakistan.}
\author{BARBU BERCEANU$^{1}$,\,\, ABDUL RAUF NIZAMI$^{2}$}
\address{$^{1}$Simion Stoilow Institute of Mathematics,
Bucharest-Romania(permanent address) and Abdus Salam School of Mathematical
Sciences, GC University, Lahore-Pakistan.}
\email{Barbu.Berceanu@imar.ro}
\address{$^{2}$University of Education, Lahore-Pakistan.}
\email{arnizami@ue.edu.pk}
\maketitle

\begin{abstract}
Using a simple recurrence relation we give a new method to compute Jones
polynomials of closed braids: we find a general expansion formula and a
rational generating function for Jones polynomials. The method is used to
estimate degree of Jones polynomials for some families of braids and to
obtain general qualitative results.
\end{abstract}

\pagestyle{myheadings} 
\markboth{\centerline {\scriptsize
B. BERCEANU\,\,\, A. R. NIZAMI }} {\centerline {\scriptsize
Recurrence Relation for Jones Polynomials }}


\section{Introduction}

\label{sec1} The \emph{Jones polynomial} $V_{L}(q)$ of an oriented link $L$
is a Laurent polynomial in the variable $\sqrt{q}$ satisfying the skein
relation 
\begin{equation*}
q^{-1}V_{L_{+}}-qV_{L_{-}}=(q^{1/2}-q^{-1/2})V_{L_{0}},
\end{equation*}
and such that the value of the unknot is 1 (see~\cite{Jones:85},\cite%
{Jones:05},\cite{Lickorish:97}). The relation holds for any oriented links
having diagrams which are identical except near one crossing where they
differ as in the figure below:

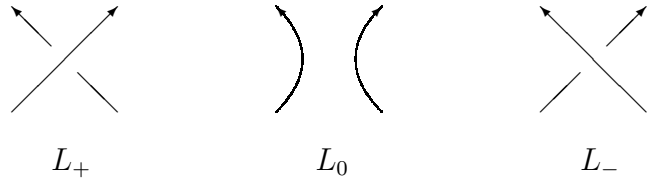
\begin{center}
\begin{picture}(360,100)
\put(75,18){$L_+$} \put(175,18){$L_0$} \put(275,18){$L_-$}
\put(60,40){\vector(1,1){40}}
\put(300,40){\vector(-1,1){40}}
\put(75,65){\vector(-1,1){15}}
\put(285,65){\vector(1,1){15}}
\put(163,77){\vector(-1,1){3}}
\put(197,77){\vector(1,1){3}}
\put(100,40){\line(-1,1){15}}
\put(260,40){\line(1,1){15}}
\qbezier(160,40)(180,60)(160,80)
\qbezier(200,40)(180,60)(200,80)
\end{picture}
\end{center}

\noindent Any link $L$ can be obtained as a closure of a braid $\beta\in 
\mathcal{B}_{n}$ (for some $n$), $L=\widehat{\beta}$. We will use classical
Artin presentation of braids (\cite{Artin:47},\cite{Bir:74}) with generators 
$x_{1},\dots,x_{n-1}$, where $x_{i}$ is: 

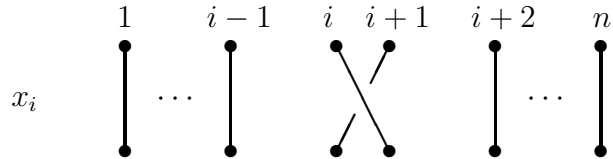
\begin{center}
\begin{picture}(360,70)
\thicklines
 \put(90,60){$1$}
 \put(90,10){$\bullet$}
  \put(93,13){\line(0,1){40}}
  \put(90,50){$\bullet$}
\put(270,60){$n$}
  \put(270,10){$\bullet$}
 \put(273,13){\line(0,1){40}}
  \put(270,50){$\bullet$}
\put(124,60){$i-1$}
   \put(130,10){$\bullet$}
    \put(133,13){\line(0,1){40}}
    \put(130,50){$\bullet$}
\put(224,60){$i+2$}
     \put(230,10){$\bullet$}
     \put(233,13){\line(0,1){40}}
      \put(230,50){$\bullet$}
     \put(168,60){$i$}
       \put(170,10){$\bullet$}
       \put(170,50){$\bullet$}
       \put(184,60){$i+1$}
         \put(190,10){$\bullet$}
         \put(190,50){$\bullet$}
 \put(193,13){\line(-1,2){20}}
  \put(173,13){\line(1,2){7}}
  \put(193,53){\line(-1,-2){7}}
\put(105,30){$\cdots$}
 \put(245,30){$\cdots$}

 \put(50,30){$x_{i}$}
\end{picture}
\end{center}

\noindent Fixing the sequence ($x_{i_{1}},\ldots,x_{i_{k}}$) of generators, $%
i_h\in\{1,2,\ldots,n-1\}$, and all the exponents $a_{1},\ldots,a_{k}$ ($%
a_h\in\mathbb{Z}$), but the $j$-th exponent is variable, we have the braid $%
\beta(e)=x_{i_{1}}^{a_{1}}\dots x_{i_{j}}^{e}\dots x_{i_{k}}^{a_{k}}\in 
\mathcal{B}_{n}$ and its Jones polynomial $V_{n}(e)=V(\widehat{\beta(e)})$;
in general $V_{n}(\beta)$ stands for $V(\widehat{\beta})$, where $\beta\in 
\mathcal{B}_{n}$. We will freely use Artin braid relations and Markov moves
in proofs and some computations.

We also change the variable $s=q^{-1/2}$ in order to obtain, for large $e$,
polynomials in $s$ (and not Laurent polynomials in $\sqrt{q}$). Our first
result is:

\begin{theorem}[Recurrence relation]
\label{thm1.1} For any $e\in \mathbb{Z}$ we have 
\begin{equation*}
V_{n}(e+2)=(s^{3}-s)V_{n}(e+1)+s^{4}V_{n}(e)\, .
\end{equation*}
\end{theorem}

This formula shows that in computations with Jones polynomial of braids the
exponents can be reduced to 0 and 1 and is nothing new here: this comes from
quadratic relations in Hecke algebras and V.F.R. Jones and A. Ocneanu traces 
\cite{Jon},\cite{Ocn}. See also \cite{Mor} for applications of these ideas
to computations.

Systematic and elementary algebraic consequences of quadratic reduction
gives us a general expansion formula and the generating function for Jones
polynomials. Let us introduce two basic polynomials $%
P_{0}^{[a]}(s)=s^{3a}+(-1)^{a}s^{a+2}$ and $%
P_{1}^{[a]}(s)=s^{3a-1}+(-1)^{a+1}s^{a-1}$ (if $a$ is not positive, these
are Laurent polynomials).

\begin{theorem}[Expansion formula]
\label{thm1.2} The following formula holds for braids in $\mathcal{B}_{n}$ $%
(a_{1},\dots,a_{k}\in \mathbb{Z}$ and $J_{*}=(j_{1},\dots,j_{k})):$ 
\begin{equation*}
V_{n}(x_{i_{1}}^{a_{1}}x_{i_{2}}^{a_{2}}\dots x_{i_{k}}^{a_{k}})(s) =\frac{1%
} {(s^{2}+1)^{k}}\sum\limits_{J_{*}\in\{0,1\}^{k}}P_{j_{1}}^{[a_{1}]}(s)%
\dots P_{j_{k}}^{[a_{k}]}(s)V_n(x_{i_{1}}^{j_{1}}\dots
x_{i_{k}}^{j_{k}})(s)\,.
\end{equation*}
\end{theorem}

We define the \emph{generating function} for Jones polynomials corresponding
to the braids $\beta ^{A_{\ast }}=x_{i_{1}}^{a_{1}}\dots
x_{i_{k}}^{a_{k}}\in \mathcal{B}_{n}$ with a fixed sequence $I_{\ast
}=(i_{1},\dots ,i_{k})\in \{1,\dots ,n-1\}^{k}$, as a formal series in $%
t_{1},\dots ,t_{k}$ 
\begin{equation*}
\mathcal{V}_{n,I_{\ast }}(t_{1},\dots ,t_{k})=\sum_{\substack{ A_{\ast }\in 
\mathbb{Z}^{k}}}V_{n}(x_{i_{1}}^{a_{1}}\dots
x_{i_{k}}^{a_{k}})t_{1}^{a_{1}}\dots t_{k}^{a_{k}}.
\end{equation*}

In the next formula we use new polynomials in the formal variable $t$: $%
q(t)= $ $=(1-s^{3}t)(1+st)$, $Q_{0}(t)=1-(s^{3}-s)t$, and $Q_{1}(t)=t$.

\begin{theorem}
\label{thm1.3} The generating function of Jones polynomials of type $(n,I_*)$
is a rational function in $t_{1},\dots,t_{k}$ given by 
\begin{equation*}
\mathcal{V}_{n,I_{*}}(t_{1},\dots,t_{k}) =\frac{1}{q(t_{1})\dots q(t_{k})}%
\sum_{\substack{ J_{*}\in \{0,1\}^{k}}}Q_{j_{1}}(t_{1})\dots
Q_{j_{k}}(t_{k})V_n(x_{i_{1}}^{j_{1}}\dots x_{i_{k}}^{j_{k}})\,.
\end{equation*}
\end{theorem}

In Section~\ref{sec2} we give proofs for Theorems~$1, 2$ and $3$. The
necessary algebraic background is given in Appendix (for details see \cite%
{nizami:08}).

In Section~\ref{sec3} we will use recurrence relation to evaluate the degree
of Jones polynomial. One result is the following:

\begin{proposition}
\label{prop1.4} If $e\gg 0$, then $V_{n}(e)$ is a polynomial in $s$ and

\begin{equation*}
\lim_{e \rightarrow \infty} \deg V_{n}(e)=+\infty.
\end{equation*}
\end{proposition}

There are still open problems relating Jones invariant with closed 3-braids,
see J. Birman paper \cite{Bir}. In Section~\ref{sec4} we compute Jones
polynomials of 2-braids and some families of $3$-braids and powers of
Garside braid $\Delta_{3}=x_{1}x_{2}x_{1} $ (see~\cite{Garside:69},~\cite%
{Bir:74}) and we establish the following results:

\begin{proposition}
\label{prop1.5} For any $k\geq0$, we have 
\begin{equation*}
V_{3}(\Delta_{3}^{2k})=2s^{12k}+s^{6k+2}+s^{6k-2},
\end{equation*}
\begin{equation*}
V_{3}(\Delta_{3}^{2k+1})=-s^{6k+5}-s^{6k+1}.
\end{equation*}
\end{proposition}

\begin{theorem}
\label{thm1.6} \textbf{a\emph{)}} For exponents $a_{i}\geq 1$ $(i=1,\dots
,2L),$ Jones polynomial of the 3-braid $\beta ^{A_{\ast
}}=x_{1}^{a_{1}}x_{2}^{a_{2}}\dots x_{1}^{a_{2L-1}}x_{2}^{a_{2L}}$ of total
degree $D=\sum a_{i}$ satisfies: 
\begin{equation*}
\deg V_{3}(x_{1}^{a_{1}}x_{2}^{a_{2}}\dots
x_{1}^{a_{2L-1}}x_{2}^{a_{2L}})\leq 3D-2L.
\end{equation*}%
\textbf{b\emph{)}} For exponents $a_{i}\geq 2$ $(i=1,\ldots ,2L)$ and $k\leq
4$ the equality holds: 
\begin{equation*}
\deg V_{3}(x_{1}^{a_{1}}x_{2}^{a_{2}}\dots
x_{1}^{a_{2L-1}}x_{2}^{a_{2L}})=3D-2L
\end{equation*}%
and the leading coefficient of $V_{3}($\ $\beta ^{A_{\ast }})$ is $1$.
\end{theorem}

The degree, coefficients and breadth of Jones polynomial are well understood
for special classes of links (\cite{Mur},\cite{Thit}). See also \cite{Sto}
for a study of this subject. Research in this area (\cite{Big},\cite{Eli},%
\cite{Dar}) is motivated by a natural question: are there nontrivial
solutions for the equation $V(\widehat{\beta })=1$ ? In Section~\ref{sec5}
we find that:

\begin{theorem}
\label{thm1.7} The sequence $(V_n(e))_{e\in \mathbb{Z}}$ could contain at
most two polynomials $V_n(a)$ and $V_n(b)$ equal to $1$, and in this case $%
|a-b|=2$.
\end{theorem}


\section{Proofs of the main results}

\label{sec2} In this section we fix $n\geq2$, generators $x_{i_{1}},\dots,
x_{i_{k}}\in \mathcal{B}_{n}$, the index $j$, and the exponents $%
(a_{1},\dots,a_{j-1},a_{j+1},\dots,a_{k})$. First we translate the skein
relation for Jones polynomials into a recurrence relation of Fibonacci type: 
\begin{equation*}
V_{n}(e+2)=(s^{3}-s)V_{n}(e+1)+s^{4}V_{n}(e).
\end{equation*}

\begin{proof}[Proof of Theorem~\protect\ref{thm1.1}]
Let $\gamma =\alpha x_{i_{j}}^{e}\beta $, where $\alpha
=x_{i_{1}}^{a_{1}}\dots x_{i_{j-1}}^{a_{j-1}}$ and $\beta
=x_{i_{j+1}}^{a_{j+1}}\dots x_{i_{k}}^{a_{k}}$ be fixed. Let us take $e$
positive. Since the geometrical change appears only at the $j$-th position
we draw a local picture (for the $j$-th factor only):\newline

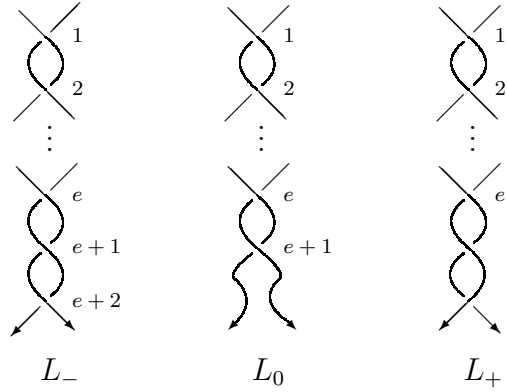
\begin{picture}(360,160)
 \qbezier(80,140)(92,130)(82,122)
 \qbezier(78,138)(68,130)(80,120)

 \put(80,140){\line(-1,1){10}}
 \put(82,142){\line(1,1){11}}
 \put(80,120){\line(1,-1){11}}
  \put(78,118){\line(-1,-1){10}}
\put(90,138){{\tiny 1}}
 \put(90,118){{\tiny 2}}
\qbezier(160,140)(172,130)(162,122)
 \qbezier(158,138)(148,130)(160,120)
 \put(160,140){\line(-1,1){11}}
 \put(160,120){\line(1,-1){11}}
 \put(162,142){\line(1,1){10}}
  \put(158,118){\line(-1,-1){10}}
 \put(170,138){{\tiny 1}}
 \put(170,118){{\tiny 2}}
\qbezier(240,140)(252,130)(242,122)
 \qbezier(238,138)(228,130)(240,120)
 \put(240,140){\line(-1,1){11}}

 \put(240,120){\line(1,-1){11}}
 \put(242,142){\line(1,1){10}}
  \put(238,118){\line(-1,-1){10}}
 \put(250,138){{\tiny 1}}
 \put(250,118){{\tiny 2}}
\qbezier(80,80)(92,70)(82,62)
 \qbezier(78,78)(68,70)(80,60)
 \qbezier(80,60)(92,50)(82,42)
 \qbezier(78,58)(68,50)(80,40)
 \put(80,80){\line(-1,1){11}}
 \put(82,82){\line(1,1){10}}
\put(80,40){\vector(1,-1){11}}
 \put(78,38){\vector(-1,-1){11}}
\put(90,78){{\tiny $e$}}
 \put(90,58){{\tiny $e+1$}}
  \put(90,38){{\tiny $e+2$}}
\put(160,80){\line(-1,1){11}}
 \put(162,82){\line(1,1){10}}
 \qbezier(160,80)(172,70)(162,62)
 \qbezier(158,78)(148,70)(160,60)

 \qbezier(160,60)(172,50)(168,48)
 \qbezier(168,48)(160,40)(172,32)
 \put(172,32){\vector(1,-1){3}}

 \qbezier(158,58)(148,50)(152,48)
 \qbezier(152,48)(160,40)(152,32)
 \put(152,32){\vector(-1,-1){3}}
\put(170,78){{\tiny $e$}}
 \put(170,58){{\tiny$e+1$}}
\put(240,80){\line(-1,1){11}} \put(242,82){\line(1,1){10}}
\qbezier(240,80)(252,70)(242,62)
  \qbezier(238,78)(228,70)(240,60)

    \qbezier(240,60)(252,50)(240,40)
     \qbezier(238,58)(228,50)(238,42)
     \put(240,40){\vector(-1,-1){11}}
     \put(242,38){\vector(1,-1){10}}
     \put(250,78){{\tiny $e$}}
     \put(78,10){$L_{-}$}\put(158,10){$L_{0}$}\put(238,10){$L_{+}$}
     \put(80,97){\vdots}  \put(160,97){\vdots}  \put(240,97){\vdots}
\end{picture}

Now it is clear from the figure that the relation $%
q^{-1}V_{L_{+}}-qV_{L_{-}}=(q^{1/2}-q^{-1/2})V_{L_{0}}$ becomes, after the
changes $q\rightarrow s^{-2}$ and $L_{-}\rightarrow \widehat{\beta }%
(e+2),\,\,L_{0}\rightarrow \widehat{\beta }(e+1)$ and $L_{+}\rightarrow 
\widehat{\beta }(e)$, simply $V_{n}(e+2)=(s^{3}-s)V_{n}(e+1)+s^{4}V_{n}(e)$.
The negative case ($e<0$) can be reduced to the positive case using the
transformation $V_{n}(x_{i_{1}}^{a_{1}}\dots x_{i_{j}}^{e}\dots
x_{i_{k}}^{a_{k}})=V_{n}(x_{i_{1}}^{a_{1}}\dots
x_{i_{j}}^{m+e}x_{i_{j}}^{-m}\dots x_{i_{k}}^{a_{k}})$ with $m$ big enough.
\end{proof}

\begin{remark}
\label{rem2.1} We can find a similar recurrence relation for
Alexander-Conway polynomial and, in general, for HOMFLY polynomial. It is
interesting to remark that in the classical cases, Jones and Alexander, the
roots of the characteristic equation are rational functions: $r_{1}=-s$, $%
r_{2}=s^{3}$ for Jones polynomial and $r_{1}=s$, $r_{2}=-s^{-1}$ for
Alexander polynomial. Jones recurrence is nicer because the roots are
polynomials. For Alexander-Conway polynomial and more results on the roots
of the characteristic equation for HOMFLY polynomial, see~\cite%
{Barbu-Rehana:09}.
\end{remark}

In Appendix (or see~\cite{nizami:08} for full details), multiple Fibonacci
sequences are introduced. Our main example is the multiple Fibonacci
sequence given by Jones polynomials of closures of braids $V_{n,I_{\ast
}}(x_{i_{1}}^{a_{1}}x_{i_{2}}^{a_{2}}\dots x_{i_{k}}^{a_{k}})$, where $%
a_{1},\dots ,a_{k}\in \mathbb{Z}$; we fix $n$ (all braids are in $\mathcal{B 
}_{n}$), $k$, and also $I_{\ast }=(i_{1},\dots ,i_{k})$ with indices $%
i_{h}\in \{1,\dots ,n-1\}$. Applying Theorem~\ref{thm6.1} we obtain:

\begin{proof}[Proof of Theorem~\protect\ref{thm1.2}]
From the basic recurrence relation, we have $r^{2}=(s^{3}-s)r+s^{4},$ with
the roots $r_{1}=-s$ and $r_{2}=s^{3}$. Hence $D=s^{3}+s=s(s^{2}+1)$, $%
S_{0}^{[n]}=s^{3n+1}+(-1)^{n}s^{n+3}=s[s^{3n}+(-1)^{n}s^{n+2}]$, $%
S_{1}^{[n]}=s^{3n}+(-1)^{n+1}s^{n}=s[s^{3n-1}+(-1)^{n+1}s^{n-1}].$ Now, let
us introduce two new sequences of polynomials (if $n\geq1$), $%
P_{0}^{[n]}(s)=s^{3n}+(-1)^{n}s^{n+2}$ and $%
P_{1}^{[n]}(s)=s^{3n-1}+(-1)^{n+1}s^{n-1}$. Using Theorem~\ref{thm6.1} we
get 
\begin{equation*}
V_{n}(x_{i_{1}}^{a_{1}}x_{i_{2}}^{a_{2}}\dots x_{i_{k}}^{a_{k}}) =\frac{1} {%
s^{k}(s^{2}+1)^{k}}\sum\limits_{J_{*}\in\{0,1\}^{k}}(sP_{j_{1}}^{[a_{1}]})
(sP_{j_{2}}^{[a_{2}]})\dots (sP_{j_{k}}^{[a_{k}]})V_n(x_{i_{1}}^{j_{1}}\dots
x_{i_{k}}^{j_{k}})
\end{equation*}
\begin{equation*}
\,\,\,\,\,\,\,\,\,\,=\frac{1}{(s^{2}+1)^{k}}\sum\limits_{J_{*}\in\{0,1
\}^{k}}P_{j_{1}}^{[a_{1}]}\dots
P_{j_{k}}^{[a_{k}]}V_n(x_{i_{1}}^{j_{1}}\dots x_{i_{k}}^{j_{k}})\,.
\end{equation*}
\end{proof}

\begin{proof}[Proof of Theorem~\protect\ref{thm1.3}]
In the second part of Theorem~\ref{thm6.1} take $r_{1}=-s$ and $r_{2}=s^{3}$.
\end{proof}

\begin{remark}
\label{rem2.3} \textbf{\emph{a)}} The expansion formula has $2^{k}$ terms,
where $k$ is the number of factors of the braid $\beta
=x_{i_{1}}^{a_{1}}\dots x_{i_{k}}^{a_{k}}$, less than $2^{|a_{1}|+\ldots
+|a_{k}|}$, the number of terms in Kauffman expansion (see ~\cite%
{Kauffman:87}, ~\cite{Lickorish:97}).

\textbf{\emph{b)}} This formula reduces the computation of Jones polynomial
of $\beta $ to the computation of Jones polynomial of (simpler) positive
braids $\beta _{J}=x_{i_{1}}^{j_{1}}\dots x_{i_{k}}^{j_{k}}$, with $J_{\ast
}\in \{0,1\}^{k}$. If these braids contain exponents $\geq 2$ (after
possible concatenations the number of factors is less than $k$), another
application of expansion formula will reduce degrees and the number of
factors. Therefore, iterated application of the expansion formula reduces
computation to Jones polynomials of Markov square-free braids (see \cite%
{Barbu-Rehana:10} for definition and proofs) $\beta
_{I}=x_{i_{1}}x_{i_{2}}\dots x_{i_{k}}$, where $1\leq i_{1}<i_{2}<\dots
<i_{k}\leq n-1$, and in this case $V_{n}(\beta _{I})=(-s-s^{-1})^{n-k-1}$.
\end{remark}

\section{Degree of Jones polynomial}

\label{sec3} First we study the behavior of the function $e \mapsto \deg
V_n( \widehat{\beta (e)})=\deg V(e)$ for an $n$-braid $%
\beta(e)=x_{i_{1}}^{a_{1}} \dots x_{i_{j}}^{e}\dots x_{i_{k}}^{a_{k}}$. More
precise results are given for some families of 3-braids in Section~\ref{sec4}%
. If the Laurent polynomial $f=a_qs^q+a_{q+1}s^{q+1}+\dots +a_ps^p $ has
coefficients $a_q,a_p \neq 0 $, we denote the degree $\mathrm{deg}$$(f)=p $,
the order $\mathrm{ord}$$(f)=q $, and the leading coefficient $\mathrm{coeff}
$$(f)=a_p$; we also use the convention $\mathrm{deg}(0)\leq 0$.

\begin{definition}
\label{def3.1} $a$)\,\, $\beta(e)$, $\beta(e+1)$ is a \emph{\textbf{stable
pair}} if 
\begin{equation*}
\deg V(e+1)>1+\deg V(e).
\end{equation*}
$b$)\,\, $\beta(e)$, $\beta(e+1)$ is a \emph{\textbf{semistable pair}} if 
\begin{equation*}
\deg V(e+1)\leq \deg V(e).
\end{equation*}

$c$)\,\, $\beta(e)$, $\beta(e+1)$ is a \emph{\textbf{critical pair}} if 
\begin{equation*}
\deg V(e+1)=1+\deg V(e).
\end{equation*}
\end{definition}

\begin{proposition}[stable case]
\label{prop3.2} If $\beta(e)$, $\beta(e+1)$ is a stable pair then for all $%
m\in\mathbb{N}$ the pair $\beta(e+m)$, $\beta(e+m+1)$ is also stable, and
for $m\geq1$\,\,\,\, 
\begin{equation*}
\deg V(e+m)=\deg V(e+1)+3(m-1)
\end{equation*}
and $\beta(e+1)$ and $\beta(e+m)$ have the same leading coefficient.
\end{proposition}

\begin{proof}
We prove the statement by induction on $m$. Hypothesis gives the case $m=0$.
If $\beta (e+m-1)$, $\beta (e+m)$ is a stable pair, then Theorem~\ref{thm1.1}
implies 
\begin{equation}
V(e+m+1)=(s^{3}-s)V(e+m)+s^{4}V(e+m-1)  \label{eq2.1}
\end{equation}%
and $\deg V(e+m)>1+\deg V(e+m-1)$ implies $\deg V(e+m+1)=\deg V(e+m)+3=\deg
V(e+1)+3m$ (hence $\beta (e+m)$, $\beta (e+m+1)$ is stable) and
coeff\thinspace $V(e+m+1)$\thinspace \thinspace =\thinspace \thinspace
coeff\thinspace $V(e+m)$.
\end{proof}

\begin{proposition}[semistable case]
\label{prop3.3} If $\beta(e)$, $\beta(e+1)$ is a semistable pair then for
all $m\geq 1$ the pair $\beta(e+m)$, $\beta(e+m+1)$ is stable, and for $%
m\geq2 $\, , 
\begin{equation*}
\deg V(e+m)=\deg V(e)+3m-2
\end{equation*}
and also $\beta(e)$ and $\beta(e+m)$ have the same leading coefficient.
\end{proposition}

\begin{proof}
The pair $\beta(e)$, $\beta(e+1)$ is semistable and the recurrence 
\begin{equation*}
V(e+2)=(s^{3}-s)V(e+1)+s^{4}V(e)
\end{equation*}
implies $\deg V(e+2)=\deg V(e)+4>\deg V(e+1)+1$, hence the pair $\beta(e+1)$%
, $\beta(e+2)$ is stable and we can apply Proposition~\ref{prop3.2}.
\end{proof}

In the last (critical) case we cannot obtain complete results; see the Case 2%
$c$) below.

\begin{proposition}[critical case]
\label{prop3.4} Let $\beta(e)$, $\beta(e+1)$ be a critical pair and $C$ the
sum of leading coefficients of $\beta(e)$ and $\beta(e+1)$.\newline

\indent\textbf{Case\,\emph{1)}} If $C\neq0$, then for all $m\geq1$ the pair $%
\beta(e+m)$, $\beta(e+m+1)$ is stable, and 
\begin{equation*}
\deg V(e+m)=\deg V(e+1)+3(m-1)\, ,
\end{equation*}
and also the leading coefficient of $V(e+m)$ is $C$.\newline
\indent\textbf{Case\,\emph{2)}} If $C=0$ and\newline
\indent \textbf{2a\emph{)}} $\deg V(e+2)=2+\deg V(e+1)$, then for all $%
m\geq1 $ the pair $\beta(e+m)$, $\beta(e+m+1)$ is stable and for all $m\geq2$
the following holds 
\begin{equation*}
\deg V(e+m)=\deg V(e+1)+3m-2.
\end{equation*}
\indent \textbf{2b\emph{)}} $\deg V(e+2)\leq \deg V(e+1)$, then $\beta(e+2)$%
, $\beta(e+3)$ is a stable pair and for all $m\geq2$ the pair $\beta(e+m)$, $%
\beta(e+m+1)$ is stable. For all $m\geq3$ we have the relation 
\begin{equation*}
\deg V(e+m)=\deg V(e+1)+3m-5.
\end{equation*}
\indent\textbf{2c\emph{)}} $\deg V(e+2)=1+\deg V(e+1)$, then the pair $%
\beta(e+1)$, $\beta(e+2)$ is critical.
\end{proposition}

\begin{proof}
Case\thinspace 1) We prove it by induction on $m$. For $m=1,$ Theorem~\ref%
{thm1.1} implies 
\begin{equation*}
V(e+2)=(s^{3}-s)V(e+1)+s^{4}V(e),
\end{equation*}%
and condition $C\neq 0$ implies that $V(e+2)$ has degree $\deg V(e+1)+3$ and
leading coefficient $C$. We can apply Proposition~\ref{prop3.2} for the
stable pair $\beta (e+1)$, $\beta (e+2)$.\newline
\indent Case\thinspace 2) The hypotheses imply that $\beta (e+1)$, $\beta
(e+2)$ is $a$) stable, $b$) semistable, and $c$) critical, respectively.
\end{proof}

\begin{remark}
In the critical case the degree of $V(e+2)$ can be arbitrarily small
compared to degree of $V(e)$; see, in Section~\ref{sec4}, the family $%
V_{3}(a_{1},1,3,1)$ and also the degrees in the sequence $V_{\Delta}(6k+1)$, 
$V_ {\Delta} (6k+2) $, $V_ {\Delta} (6k+3) $.
\end{remark}

The order of $V(e)$ is increasing for large value of $e$.

\begin{proposition}
\label{prop3.5} With the same notations we have:\newline
\indent\textbf{a\emph{)}}\thinspace\ $\emph{ord}\,V(e+2)\geq 1+\min (\emph{%
ord}\,V(e),\emph{ord}\,V(e+1)).$\newline
\indent\textbf{b\emph{)}} For any $m\geq 2$,\thinspace \thinspace\ $\emph{ord%
}\,V(e+m)\geq \min (\emph{ord}\,V(e),\emph{ord}\,V(e+1))+(m-1).$\newline
In particular, if $e\gg 0$, then $V(e)$ is a polynomial in $s$.
\end{proposition}

\begin{proof}
$\mathbf{\emph{a}})$ is a direct consequence of Theorem~\ref{thm1.1} and $b)$
comes from \emph{a}) by induction.
\end{proof}

\begin{proof}[Proof of Proposition~\protect\ref{prop1.4}]
First choose $e_{0}\geq e$ such that for any $m\geq e_{0}$, $V(m)$ is a
polynomial in $s$, as in the previous proposition. Let us suppose that in
the sequence $\beta (e_{0}),\beta (e_{0}+1),\dots $ there is a stable (or
semistable) pair, say $\beta (e_{1}),\beta (e_{1}+1)$. Then according to
Proposition~\ref{prop3.2} (or Proposition~\ref{prop3.3} ), all the
consecutive pairs in the sequence $\beta (e_{1}+1),\beta (e_{1}+2),\dots $
are stable. Therefore $\deg V(e_{1}+m)\rightarrow \infty $ as $m\rightarrow
\infty $. If any pair of the sequence $\beta (e_{0}),\beta (e_{0}+1),\dots $
is critical, then the sequence $\deg \beta (e_{0}),\deg \beta
(e_{0}+1),\dots $ is arithmetic and again we get the desired result. \newline
\emph{Second Proof}. We can use Proposition~\ref{prop3.5} along with the
fact that Jones polynomial cannot be $0$.
\end{proof}

Similar results can be proved for negative exponents.

\begin{proposition}
\label{prop3.6} With $e<0$ we have:\newline
\indent\textbf{a\emph{)}}\, $\deg V(e-2)\leq \max(\deg V(e-1),\deg V(e))-1.$%
\newline
\indent\textbf{b\emph{)}} For any $m\geq2,$\, $\deg V(e-m)\leq \max(\deg
V(e-1),\deg V(e))-(m-1).$\newline
In particular, for $e\ll0$, $V(e)$ is a polynomial in $s^{-1}$ and 
\begin{equation*}
\lim_{e\rightarrow -\infty} \deg V(e)= -\infty.
\end{equation*}
\end{proposition}

\begin{proof}
Use the recurrence relation in the form $%
V(e-2)=(s^{-3}-s^{-1})V(e-1)+s^{-4}V(e).$
\end{proof}


\section{Examples}

\label{sec4} Three types of examples are given: arbitrary braids in $%
\mathcal{B}_{2}$, braids $x_{1}^{a_{1}}x_{2}^{a_{2}}\dots x_{2}^{a_{2L}}$ in 
$\mathcal{B}_{3}$ (with complete results for $L\leq 2$) and powers of
Garside braid $\Delta _{3}$. Some of the results are well known, especially
those which are connected with torus links, some seem to be new, and all of
them show how to use the recurrence relation.

\begin{proposition}
\label{prop4.1} Let $x_{1}^{a}$, $a\in\mathbb{Z}$, be a braid in $\mathcal{B}%
_{2}$, and $\widehat{x_{1}^{a}}$ the corresponding link. Its Jones
polynomial is given by: for $a\leq-2$ 
\begin{equation*}
V_{2}(a)=-s^{3a+1}+s^{3a+3}-s^{3a+5}+\dots
-(-1)^{a+1}s^{a-5}+(-1)^{a+1}s^{a-3}+(-1)^{a+1}s^{a+1},
\end{equation*}
for the next three values 
\begin{equation*}
V_2(-1)=1,\, V_2(0)=-s-s^{-1},\, V_2(1)=1,
\end{equation*}
and for $a\geq2$ 
\begin{equation*}
V_{2}(a)=-s^{3a-1}+s^{3a-3}-s^{3a-5}+\dots
-(-1)^{a+1}s^{a+5}+(-1)^{a+1}s^{a+3}+(-1)^{a+1}s^{a-1}.
\end{equation*}
\end{proposition}

\begin{proof}
The Jones polynomials of the trivial two-component link and trivial knot are
given by 
\begin{equation}
V_{2}(0)=-s-s^{-1},V_{2}(\pm 1)=1.  \label{eq3.1}
\end{equation}%
From the basic recurrence relation 
the general term is 
\begin{equation}
V_{2}(a)=\left( \frac{-s}{1+s^{2}}\right) (s^{3})^{a}+\left( \frac{%
-1-s^{2}-s^{4}}{s+s^{3}}\right) (-s)^{a}\,.  \label{eq3.4}
\end{equation}%
Elementary computations give the desired result.

For $a\leq -2$, the coefficients $\frac{-s}{1+s^{2}}$ and $\frac{%
-1-s^{2}-s^{4}}{s+s^{3}}$ are invariant under $s\rightarrow s^{-1}$,
therefore . 
\begin{equation}
V_{2}(a)(s)=\left( \frac{-s}{1+s^{2}}\right) (s^{3})^{a}+\left( \frac{%
-1-s^{2}-s^{4}}{s+s^{3}}\right) (-s)^{a}=V_{2}(-a)(s^{-1})\,,  \label{eq3.44}
\end{equation}%
and this proves the formula for negative exponents.
\end{proof}

\begin{proposition}
\label{prop4.2} Let $\alpha,\beta \in\mathcal{B}_{n}$, $\gamma=\alpha\beta$,
and $\widetilde{\gamma_{k}}=\alpha x_{n}^{k}\beta \in\mathcal{B}_{n+1}$ $%
(k\in \mathbb{Z})$, then 
\begin{equation*}
V_{n+1}(\widetilde{\gamma_{k}})=V_{n}(\gamma)V_{2}(x_{1}^{k}).
\end{equation*}
\end{proposition}

\begin{proof}
Define $f(k)=V_{n+1}(\widetilde{\gamma _{k}})$ and $g(k)=V_{n}(\gamma
)V_{2}(x_{1}^{k})$. These coincide for $k=0$: $g(0)=V_{n}(\gamma
)V_{2}(1)=V_{n}(\gamma )(-s-s^{-1})$ and $f(0)=V_{n+1}(\widetilde{ \gamma
_{0}})=V_{n+1}(\gamma )= V_{n}(\gamma )(-s-s^{-1}) $, and $k=1$: $%
g(1)=V_{n}(\gamma )V_{2}(x_{1})=V_{n}(\gamma )$ and $f(1)=V_{n+1}(\alpha
x_{n}\beta )=V_{n+1}(\beta \alpha x_{n})=V_{n}(\beta \alpha )=V_{n}(\alpha
\beta )=V_{n}(\gamma )$; also $f(k)$ and $g(k)$ satisfy the same recurrence
relation.
\end{proof}

\begin{corollary}
\label{prop4.3} $%
V_{3}(x_{1}^{a_{1}}x_{2}^{a_{2}})=V_{2}(x_{1}^{a_{1}})V_{2}(x_{1}^{a_{2}}).$
\end{corollary}

Now we compute the degree and the leading coefficient for Jones polynomial $%
V_{3}(a_{1},a_{2},a_{3},a_{4})=V_{3}(x_{1}^{a_{1}}x_{2}^{a_{2}}x_{1}^{a_{3}}x_{2}^{a_{4}}) 
$, where $a_{i}\geq 1$. We denote by $D$ the total degree $%
a_{1}+a_{2}+a_{3}+a_{4}$.

\begin{proposition}[generic case]
\label{prop4.4} If $a_{1},a_{2},a_{3},a_{4}\geq2$, then the leading term of
Jones polynomial $V_{3}(a_{1},a_{2},a_{3},a_{4})$ is $+s^{3D-4}$.
\end{proposition}

\begin{proof}
Using the general expansion formula we find 
\begin{equation*}
V_{3}(a_{1},a_{2},a_{3},a_{4})=\frac{1}{(s^{2}+1)^{4}}\sum\limits_{J_{\ast
}\in
\{0,1%
\}^{4}}P_{j_{1}}^{[a_{1}]}P_{j_{2}}^{[a_{2}]}P_{j_{3}}^{[a_{3}]}P_{j_{4}}^{[a_{4}]}V_{3}(j_{1},j_{2},j_{3},j_{4})\,.
\end{equation*}%
The degree and leading coefficient of $V_{3}(j_{1},j_{2},j_{3},j_{4})$ are
in Table~2 (proof of Theorem~\ref{thm1.6} contains a user guide for the
table). In the above formula maximal degree is obtained from $%
V_{3}(1,0,1,0)=V_{3}(0,1,0,1)$ (coefficient $1+1$) and from $V_{3}(1,1,1,1)$
(coefficient $-1$).
\end{proof}

\begin{proposition}
\label{prop4.5} The leading terms of Jones polynomials $%
V_{3}(a_1,a_2,a_3,a_4)$ for positive exponents are given by the next table:

\mbox{}

\begin{tabular}{l}
$\mathbf{Table\,1}.$ $V_{3}(a_{1},a_{2},a_{3},a_{4})(s)$ \\ 
\end{tabular}

\mbox{}

\begin{tabular}{|l|c|l|}
\hline
$(a_{1},a_{2},a_{3},a_{4})$ & Critical cases & \quad \mbox{} \quad \mbox{}
Stable cases \\ \hline
$(1,1,a_{3},a_{4})$ & $-$ & $a_{3},a_{4}\geq 1:$ $-s^{3D-4}+\dots $ \\ \hline
$(a_{1},1,2,1)$ & $a_{1}=2:$ $2s^{3D-6}+\dots $ & $a_{1}\geq 3:$ $%
s^{3D-6}+\dots $ \\ \hline
$(a_{1},1,3,1)$ & 
\begin{tabular}{l}
$a_{1}=3:-s^{3D-8}+\dots $ \\ 
$a_{1}=4:-s^{3D-16}+\dots $ \\ 
\end{tabular}
& $a_{1}\geq 5:s^{3D-6}+\dots $ \\ \hline
$(a_{1},1,a_{3},1)$ & $-$ & $a_{1},a_{3}\geq 4:s^{3D-8}+\dots $ \\ \hline
$(a_{1},1,3,2)$ & $-$ & $a_{1}\geq 3:-s^{3D-6}+\dots $ \\ \hline
$(a_{1},1,a_{3},a_{4})$ & $-$ & $a_{1},a_{3},a_{4}\geq 3:-s^{3D-6}+\dots $
\\ \hline
$(a_{1},a_{2},a_{3},a_{4})$ & $-$ & $a_{1},a_{2},a_{3},a_{4}\geq 2
:s^{3D-4}+\dots $ \\ \hline
\end{tabular}
\end{proposition}

\begin{proof}
For the first line we can use $%
x_{1}x_{2}x_{1}^{a_{3}}x_{2}^{a_{4}}=x_{2}^{a_{3}}x_{1}x_{2}^{a_{4}+1}\sim
x_{1}^{a_{3}+a_{4}+1}x_{2}$. In the second line we start the recurrence with 
$V_{3}(1,1,2,1)$ and $V_{3}(2,1,2,1)=V_{\Delta}(6)$, a critical case with $%
C=1$. For the family $V_{3}(a_{1},1,3,1)$ the recurrence starts with $%
V_{3}(1,1,3,1)$ and $V_{3}(2,1,3,1)$ and we obtain (a critical case) $%
V_{3}(3,1,3,1)=-s^{16}+s^{10}+s^{6}$ and $V_{3}(4,1,3,1)=-s^{11}-s^{7}$ and
this is semistable. In the generic case $V_{3}(a_{1},1,a_{3},1)$, $%
a_{1},a_{3}\geq 4$, the expansion formula has only 4 $J_*$-blocks: 
\begin{equation*}
P_{0}^{[a_{1}]}P_{0}^{[a_{3}]}(s^{6}+\dots)+\left(
P_{1}^{[a_{1}]}P_{0}^{[a_{3}]}+P_{0}^{[a_{1}]}P_{1}^{[a_{3}]}\right)
(-s-s^{-1})+P_{1}^{[a_{1}]}P_{1}^{[a_{3}]}(-s^{8}+\dots )\,.
\end{equation*}%
The case $V_{3}(a_{1},1,3,2)$ starts with the semistable pair $%
V_{3}(1,1,3,2)=-s^{17}+\dots $ and $V_{3}(2,1,3,2)=-s^{16}+\dots $ . The
line $(a_{1},1,a_{3},a_{4})$ is given by the expansion formula with 8 $J_*$%
-blocks.
\end{proof}

The missing cases $(2,1,a_3,a_4)$ and $(a_{1},1,2,a_4)$ can be reduced to
the previous list: $x_1^2x_2x_1^{a_3}x_2^{a_4}=x_1x_2^{a_3}x_1x_2^{a_4+1}%
\sim x_1^{a_3}x_2x_1^{a_4+1}x_2 $, $x_1^{a_1}x_2x_1^2x_2^{a_4}=
x_2x_1x_2^{a_1}x_1x_2^{a_4}\sim x_1^{a_1}x_2x_1^{a_4+1}x_2 $. If $a_1\geq 3$%
, $a_3\geq 4$, then $(a_{1},1,a_3,2)$ can be reduced, too: $%
x_1^{a_1}x_2x_1^{a_3}x_2^2=x_2x_1x_2^{a_1}x_1^{a_3-1}x_2^2\sim
x_1^3x_2x_1^{a_1}x_2^{a_3-1} $.

Now our purpose is to compute Jones polynomial of the braid $\alpha
(n)=x_{1}x_{2}x_{1}x_{2}...$ ($n$ factors); this sequence contains the
powers of $\Delta _3=\Delta $: $\alpha (3k)=\Delta ^{k}$. We will use the
next table where $X$ is the canonical form of $\alpha (n)$ (i.e. the
smallest word in the length-lexicografic order with $x_{1}<x_{2}$) and $Y$
is a conjugate of $X$, suitable for computations. The number of factors of
the six $Y$'s is $2k+2$.

\mbox{}

\begin{center}
\begin{tabular}{|l|l|l|}
\hline
$\,\,\,\,\,\,\,\alpha(n)$ & \,\,\,\,\,\,\,\,\,\,\,\,\,\,\,\,\,\,\,\,\,$X$ & 
\,\,\,\,\,\,\,\,\,\,\,\,\,\,\,\,\,\,\,\,\,$Y$ \\ \cline{1-3}
$\Delta^{2k}$ & $x_{1}^{2k}x_{2}x_{1}^{2}x_{2}^{2}\dots x_{1}^{2}x_{2}$ & $%
x_{1}^{2k}x_{2}x_{1}^{2}x_{2}^{2}\dots x_{1}^{2}x_{2}$ \\ 
$\Delta^{2k}x_{1}$ & $x_{1}^{2k+1}x_{2}x_{1}^{2}x_{2}^{2}\dots
x_{1}^{2}x_{2} $ & $x_{1}^{2k+1}x_{2}x_{1}^{2}x_{2}^{2}\dots x_{1}^{2}x_{2}$
\\ 
$\Delta^{2k}x_{1}x_{2}$ & $x_{1}^{2k+1}x_{2}x_{1}^{2}x_{2}^{2}\dots
x_{1}^{2}x_{2}^{2}$ & $x_{1}^{2k+1}x_{2}x_{1}^{2}x_{2}^{2}\dots
x_{1}^{2}x_{2}^{2}$ \\ 
$\Delta^{2k+1}$ & $x_{1}^{2k+1}x_{2}x_{1}^{2}x_{2}^{2}\dots x_{2}^{2}x_{1}$
& $x_{1}^{2k+2}x_{2}x_{1}^{2}x_{2}^{2}\dots x_{1}^{2}x_{2}^{2}$ \\ 
$\Delta^{2k+1}x_{2}$ & $x_{1}^{2k+2}x_{2}x_{1}^{2}x_{2}^{2}\dots
x_{2}^{2}x_{1}$ & $x_{1}^{2k+3}x_{2}x_{1}^{2}x_{2}^{2}\dots
x_{1}^{2}x_{2}^{2}$ \\ 
$\Delta^{2k+1}x_{2}x_{1}$ & $x_{1}^{2k+2}x_{2}x_{1}^{2}x_{2}^{2}\dots
x_{2}^{2}x_{1}^{2}$ & $x_{1}^{2k+4}x_{2}x_{1}^{2}x_{2}^{2}\dots
x_{1}^{2}x_{2}^{2}$ \\ \hline
\end{tabular}
\end{center}

\mbox{}

\medskip \noindent In order to simplify the notation we denote by $%
V_{\Delta}(n)=V_{3}(\alpha(n))$ Jones polynomial of the closure of $3$-braid 
$x_{1}x_{2}x_{1}x_{2}\dots$ ($n$ factors).

\begin{proposition}
\label{prop4.17} The Jones polynomials $V_{\Delta }(n)$ satisfy the
following recurrence relations: 
\begin{equation*}
\begin{array}{ll}
V_{\Delta }(2k+1) & =(s^{3}-s)V_{\Delta }(2k)+s^{4}V_{\Delta }(2k-1) \\ 
V_{\Delta }(6k+4) & =(s^{3}-s)V_{\Delta }(6k+3)+s^{4}V_{\Delta }(6k+2) \\ 
V_{\Delta }(6k+2) & =(s^{3}-s)V_{\Delta }(6k+1)+(s^{7}-s^{5})V_{\Delta
}(6k-1)+s^{8}V_{\Delta }(6k-2) \\ 
V_{\Delta }(6k) & =(s^{3}-s)[V_{\Delta }(6k-1)+s^{4}V_{\Delta
}(6k-3)+s^{8}V_{\Delta }(6k-5)]+ \\ 
& \,\,\,\,\,\,+\,s^{12}V_{\Delta }(6k-6).%
\end{array}%
\end{equation*}
\end{proposition}

\begin{proof}
The proof is by induction. For the first relation, the case $6k+5$ is given
by the basic recurrence relation and the table. We give a general proof,
using a different idea: $\alpha (2k-1)=x_{1}x_{2}\dots x_{1}$ ($2k-1$
factors), $\alpha (2k)=x_{1}x_{2}\dots x_{2}$ ($2k$ factors), $\alpha
(2k+1)=x_{1}x_{2}\dots x_{1}x_{2}x_{1}\sim x_{2}x_{1}\dots x_{1}x_{2}\sim
x_{1}x_{2}\dots x_{1}x_{2}^{2}$ and now we can apply the basic recurrence
relation.

$V_{\Delta }(6k+4)$ is also given by the basic recurrence relation with $%
V_3(x_{1}^{2k+3}x_{2}x_{1}^{2}\dots )$, $V_3(x_{1}^{2k+2}x_{2}x_{1}^{2}\dots
)$ and $V_3(x_{1}^{2k+1}x_{2}x_{1}^{2}\dots )$. For the last two relation we
have to apply Theorem~\ref{thm1.1} two or three times.

To compute $V_{\Delta }(6k+2)$ we use the basic recurrence relation twice:
first, the recurrence relation among $\alpha
(6k+2)=x_{1}^{2k+1}x_{2}x_{1}^{2}\dots x_{1}^{2}x_{2}^{2}$, $\alpha
(6k+1)=x_{1}^{2k+1}x_{2}x_{1}^{2}\dots x_{1}^{2}x_{2}^{1}$ and $%
x_{1}^{2k+1}x_{2}x_{1}^{2}\dots x_{2}^{2}x_{1}^{2}x_{2}^{0}$ which is
conjugate to $x_{1}^{2k+3}x_{2}x_{1}^{2}\dots x_{2}^{2}=\beta $. The new
braid $\beta $, $\alpha (6k-1)\sim x_{1}^{2k+2}x_{2}x_{1}^{2}\dots x_{2}^{2} 
$ and $\alpha (6k-2)\sim x_{1}^{2k+1}x_{2}x_{1}^{2}\dots x_{2}^{2}$ are
related by the basic recurrence relation.

Finally, to verify the last formula we start with $\alpha (6k)\sim
x_{1}^{2k}x_{2}x_{1}^{2}\dots x_{1}^{2}x_{2}^{1}$. Changing the under- and
over-crossing status in the second last crossing we get the braid $%
x_{1}^{2k}x_{2}x_{1}^{2}x_{2}^{2}\dots x_{1}^{1}x_{2}^{1}\sim
x_{1}^{2k+2}x_{2}x_{1}^{2}x_{2}^{2}\dots x_{1}^{2}x_{2}^{2}$ whose Jones
polynomial is $V_{\Delta}(6k-1)$. Destroying the same crossing we obtain the
braid word $x_{1}^{2k}x_{2}x_{1}^{2}x_{2}^{2}\dots x_{1}^{2}x_{2}^{3}$ with $%
2k$ letters. We denote it by $\gamma $ and we obtain $V_{\Delta
}(6k)\,=\,(s^{3}-s)V_{\Delta }(6k-1)+s^{4}V_3(\gamma ) $. To compute $%
V_3(\gamma)$, we interchange the over- and under-crossing status of the last
crossing of $\gamma$ and obtain $x_{1}^{2k}x_{2}x_{1}^{2}x_{2}^{2}\dots
x_{1}^{2}x_{2}^{2}$ whose Jones polynomial is $V_{\Delta }(6k-3)$. The
elimination of last crossing of $\gamma$ gives $%
x_{1}^{2k}x_{2}x_{1}^{2}x_{2}^{2}\dots x_{1}^{2}x_{2}^{1}$ which we denote
by $\delta$. Hence $V_3(\gamma)=(s^3-s)V_{\Delta }(6k-3)+s^4V_3(\delta) $.
The smoothing of first crossing of $\delta$ ultimately gives the relation $%
V_3(\delta)=(s^3-s)V_{\Delta }(6k-5)+s^4 V_{\Delta }(6k-6)$ and the result
follows.
\end{proof}

\begin{lemma}
\label{lem4.19} The first Jones polynomials are given by:\newline
\begin{equation*}
\begin{array}{ll}
V_{\Delta }(0) & =s^{2}+2+s^{-2} \\ 
V_{\Delta }(1) & =-s-s^{-1} \\ 
V_{\Delta }(2) & =1 \\ 
V_{\Delta }(3) & =-s^{5}-s \\ 
V_{\Delta }(4) & =-s^{8}+s^{6}+s^{2} \\ 
V_{\Delta }(5) & =-s^{11}+s^{9}-s^{7}-s^{3}.%
\end{array}%
\end{equation*}
\end{lemma}

\begin{proof}
First three are obvious; next use Proposition~\ref{prop4.17} for $V_{\Delta
}(3)$, $V_{\Delta }(4)$ and $V_{\Delta }(5)$.
\end{proof}

\begin{remark}
\label{rem4.18} The recurrence relations of Proposition \ref{prop4.17}~ give
the following recurrence matrix: 
\begin{equation*}
\left( 
\begin{array}{l}
V_{\Delta }(6k) \\ 
V_{\Delta }(6k+1) \\ 
\vdots \\ 
V_{\Delta }(6k+5) \\ 
\end{array}%
\right) =A(k)\left( 
\begin{array}{c}
V_{\Delta }(6k-6) \\ 
V_{\Delta }(6k-5) \\ 
\vdots \\ 
V_{\Delta }(6k-1) \\ 
\end{array}%
\right) \,,
\end{equation*}%
where the Jordan normal form of $A(k)$ has a nice structure:

\begin{equation*}
A(k)\sim\left( 
\begin{array}{cccc}
J_{3} & 0 & 0 & 0 \\ 
0 & s^{6} & 0 & 0 \\ 
0 & 0 & s^{12} & 0 \\ 
0 & 0 & 0 & s^{18}%
\end{array}
\right), \quad \quad J_{3}=\left( 
\begin{array}{ccc}
0 & 1 & 0 \\ 
0 & 0 & 1 \\ 
0 & 0 & 0%
\end{array}
\right).
\end{equation*}

This fact is reflected in the following very simple general formulae of
Jones polynomials $V_{\Delta }(n)$.
\end{remark}

\begin{proposition}
\label{prop4.20} For any $k\geq0$ Jones polynomials of $%
\alpha(n)=x_{1}x_{2}x_{1}\dots(n$-$times)$ are given by:

\begin{equation*}
\begin{array}{ll}
V_{\Delta }(6k) & =2s^{12k}+s^{6k+2}+s^{6k-2} \\ 
V_{\Delta }(6k+1) & =s^{12k+3}-s^{12k+1}-s^{6k+3}-s^{6k-1} \\ 
V_{\Delta }(6k+2) & =-s^{12k+4}+s^{6k+4}+s^{6k} \\ 
V_{\Delta }(6k+3) & =-s^{6k+5}-s^{6k+1} \\ 
V_{\Delta }(6k+4) & =-s^{12k+8}+s^{6k+6}+s^{6k+2} \\ 
V_{\Delta }(6k+5) & =-s^{12k+11}+s^{6k+9}-s^{6k+7}-s^{6k+3}.%
\end{array}%
\end{equation*}
\end{proposition}

\begin{proof}
The proof is by induction: the case $k=0$ is covered by Lemma~\ref{lem4.19}
and induction step can be checked with Proposition~\ref{prop4.17}.
\end{proof}

Now we analyze Jones polynomial of general positive $3$-braids $\beta
^{A_{\ast }}=x_{1}^{a_{1}}x_{2}^{a_{2}}\dots x_{2}^{a_{2L}},$ where $A_{\ast
}=(a_{1},...,a_{2L})$, $a_{i}\geq 0$; we use the short notations $%
D=\dsum\limits_{i=1}^{2L}a_{i}$, the total degree of $\beta ^{A_*} $, $Z=%
\mathrm{card}\{i|1\leq i\leq 2L,a_{i}=0\}$. As an example, $J_{\ast
}^{1,0}=(1,0,\ldots ,1,0)$ and $J_{\ast }^{0,1}=(0,1,\ldots ,0,1)$ have the
same total sum and number of zeros $J^{1,0}=J^{1,0}=L=Z$. For a positive
sequence $A_{\ast }$ and $J_{\ast }\in \{ 0,1\} ^{2L}$, the $J_{\ast }$%
-block is the Laurent polynomial in $s$ $\ P_{J_{\ast }}^{A_{\ast }}V_{
J_{\ast }}=P_{j_{1}}^{[a_{1}]}P_{j_{2}}^{[a_{2}]}\dots
P_{j_{2L}}^{[a_{2L}]}V_3(\beta ^ {J_{\ast }})$.

\begin{lemma}
\label{lem4.26} If all $a_{i}$ are positive, the inequality holds: 
\begin{equation*}
\deg P_{J_{\ast }}^{A_{\ast }}\leq 3D-J,
\end{equation*}%
with equality if and only if$\ $ $j_{i}=0$ implies $a_{i}$ $\geq 2$.
\end{lemma}

\begin{proof}
By definition, $\deg P_{0}^{[1]}\leq 0,$ $\deg P_{0}^{[a]}=3a$ if $a\geq 2,\ 
$and $\deg P_{1}^{[a]}=3a-1$ if $a\geq 1$.
\end{proof}

We start the proof of Theorem~\ref{thm1.6} with a more general result:

\begin{theorem}
\label{thm77} If the positive 3-braid $\beta ^{A_{\ast
}}=x_{1}^{a_{1}}x_{2}^{a_{2}}\cdots x_{1}^{a_{2L-1}}x_{2}^{a_{2L}}$ has
degree $D$\ and number of zeros $Z,$ then the inequality holds: 
\begin{equation*}
\deg V_{3}(\beta ^{A_{\ast }})\leq 3D-2L+2Z.
\end{equation*}
\end{theorem}

\begin{lemma}
\label{lem4.28} The inequality in Theorem~\ref{thm77} for $L-1$ factors and
the inequality in Theorem~\ref{thm1.6} for $L$ factors imply the inequality
in Theorem~\ref{thm77} for $L$ factors.
\end{lemma}

\begin{proof}
If the braid $\beta ^{A_{\ast }}=x_{1}^{a_{1}}x_{2}^{a_{2}}\dots
x_{2}^{a_{2L}}$ has all exponents $\geq 1,$ then Theorem~\ref{thm1.6} $(L)$
gives the result. If at least one exponent is zero, then $\beta ^{A_{\ast
}}= $ $\beta ^{A_{\ast }^{\prime }},$ where $A_{\ast }^{\prime }$ is a
sequence of length $2(L-1)$ obtained from $A_{\ast }$ by deletion of a zero
exponent and concatenation of its neighbors: as an example, if \ $A_{\ast
}=(2,1,0,3,4,0), $ then $A_{\ast }^{\prime }$ can be $(2,4,4,0)$\ or $%
(6,1,0,3)$\ (any choice gives the braid $x_{1}^{6}x_{2}^{4}).$\ The new
ingredients of $A_{\ast }^{\prime }$ are $D^{\prime }=D,$ $L^{\prime }=L-1,$
\ and $Z^{\prime }=Z-1$ or $Z-2$\ (if we delete one zero and concatenate two
others), therefore%
\begin{equation*}
\deg V_{3}(\beta ^{A_{\ast }})=\deg V_{3}(\beta ^{A_{\ast }^{\prime }})\leq
3D^{\prime }-2L^{\prime }+2Z^{\prime }\leq 3D-2L+2Z.
\end{equation*}
\end{proof}

\begin{proof}[Proof of Theorem~\protect\ref{thm1.6}]
{\emph{a)}}\thinspace \thinspace\ We start induction on $L\geq 1$. If $L=1$, 
$V_{3}(x_{1}^{a_{1}}x_{2}^{a_{2}})=V_{2}(x_{1}^{a_{1}})V_{2}(x_{1}^{a_{2}}),$
and we have, up to a symmetry, the next cases:

\mbox{}

\begin{tabular}{|c|c|c|}
\hline
$(a_{1},a_{2})$ & $\deg V_{3}(x_{1}^{a_{1}}x_{2}^{a_{2}})$ & $3D-2L+2Z$ \\ 
\hline
$(0,0)$ & $2$ & $2$ \\ \hline
$(1,0)$ & $1$ & $3$ \\ \hline
$(1,1)$ & $\leq 0$ & $4$ \\ \hline
$(\geq 2,0)$ & $3a_{1}$ & $3a_{1}$ \\ \hline
$(\geq 2,1)$ & $3a_{1}-1$ & $3a_{1}+1$ \\ \hline
$(\geq 2,\geq 2)$ & $3a_{1}+3a_{2}-2$ & $3a_{1}+3a_{2}-2$ \\ \hline
\end{tabular}

\mbox{}

\mbox{}

\noindent Suppose that $L\geq 2.$ According to Lemma 8 it is enough to show
that any $J_*$-block of the expansion formula for $V_{3}(\beta ^{A_{\ast }})$%
\ $(a_{i}\geq 1)$ 
\begin{equation*}
(s^{2}+1)^{2L}V_{3}(\beta ^{A_{\ast }})=\dsum\limits_{J_{\ast }\in
\{0,1\}^{2L}}\ P_{J_{\ast }}^{A_{\ast }}V_{J_{\ast }}
\end{equation*}%
has $\deg \leq 3D+2L;$ because there is no factor $P_{j}^{[0]},$ we have $%
\deg P_{J_{\ast }}^{A_{\ast }}\leq 3D-J,$\ and it is enough to show that $%
\deg V_{J_{\ast }}\leq J+2L.$\ We begin with terms having a small
contribution:

case 1: $J\leq L-1.$ In this case the $0,1$ sequence $J_{\ast }$ contains at
least two zeros which are neighbors ( $j_{2L}$\ and $j_{1}$\ are neighbors),
we delete both and obtain a new sequence $J_{\ast }^{\prime }$ of length $%
2(L-1)$ with$\ J^{\prime }=J$ and number of zeros$\ Z^{\prime }=2(L-1)-J.$
Theorem~\ref{thm77} $(L-1)$ gives $\deg V_{J_{\ast }}\leq 3J^{\prime
}-2L^{\prime }+2Z^{\prime }=J+2L-Z.$

case 2: $J=L$ but $J_{\ast }$ is different from$\ J_{\ast }^{1,0}$ and $%
J_{\ast }^{0,1}.$ This is similar with case 1 because one can find two zero
neighbors.

Now we are looking for the main contributors:

case 3: $J_{\ast }^{1,0}$ and $J_{\ast }^{0,1}.$ Their Jones polynomials
coincide with $V_{3}(x_{1}^{L})$ of degree $3L=J+2L.$

case 4: $L+1\leq J\leq 2L-1.$ Consider the new sequence $J_{\ast }^{\prime }$
obtained after deletion of all zeros and concatenation of (nonzero)
neighbors, if necessary. For this new sequence we have $J^{\prime }=J,$ $%
Z^{\prime }=0$ and deletion of $i$\ consecutive zeros reduces the length of $%
J_{\ast }$ by $i^{+}=2{\Large (}i-\left\lfloor \frac{i}{2}\right\rfloor 
{\Large )}=i$ (for$\ i$ even) and $i+1$ (for $i$ odd). Denote by $Z_{i}$ the
number of sequences of $i$ consecutive zeros; for example, if $J_{\ast
}=(1,0,1,0,0,1,1,1,1,0),$ then $J_{\ast }^{\prime }=(3,1,1,1)$ and $Z_{1}=2,$
$Z_{2}=1,$ $Z_{\geq 3}=0.$ Therefore the total number of zeros\ in $J_{\ast
}\ $is$\ Z=\dsum\limits_{i\geq 1}iZ_{i}=2L-J$ and the length of $J_{\ast
}^{\prime }$ is $2L^{\prime }=2L-\dsum\limits_{i\geq 1}i^{+}Z_{i}$, so we
can evaluate the degrees: 
\begin{equation*}
\begin{tabular}{ll}
$\deg V_{J_{\ast }}$ & $=\deg V_{J_{\ast }^{\prime }}\leq 3J^{\prime
}-2L^{\prime }=3J-2L+\dsum\limits_{i\geq 1}i^{+}Z_{i}=$ \\ 
& $=3J-2L+2\dsum\limits_{i\geq 1}iZ_{i}-2\dsum\limits_{i\geq 2}\left\lfloor 
\frac{i}{2}\right\rfloor Z_{i}\leq $ \\ 
& $\leq 3J-2L+2(2L-J)=J+2L.$%
\end{tabular}%
\end{equation*}

case 5: $J_{\ast }=(1,1,...,1).$ This is the example studied in Proposition~%
\ref{prop4.20}, with $2L$ factors and we found the general formula%
\begin{equation*}
\deg V_{\Delta}(2L)=4L=J+2L.
\end{equation*}

\emph{Proof of Theorem}~\ref{thm1.6} {\emph{b)}}\thinspace \thinspace\ The
case $L=1$ is a consequence of Proposition~\ref{prop4.1} and Corollary~\ref%
{prop4.3}. The next cases $L=2,3,4$ are given by three tables containing: $%
\delta =\deg (P_{J_{\ast }}^{A_{\ast }})-\deg (P_{111\cdots 1}^{A_{\ast }})$
(we assume $a_{i}\geq 2$), $J_{\ast }=(j_1,\dots ,j_{2L})$, $w$ is a word
conjugate to $x_{1}^{j_{1}}x_{2}^{j_{2}}x_{1}^{j_{2k-1}}x_{2}^{j_{2k}}$, $N$
is the number of positive braids in the same conjugacy class with $w$, $T$
is the leading term of $V_{3}(w)$, and $\deg =\delta +\deg (T)$. In the last
column the top degrees are in bold characters.

\mbox{}

\begin{tabular}{l}
\textbf{Table\,2}. $L=2,\, V_{J_*}= V(x_1^{a_1}x_2^{a_2}x_1^{a_3}x_2^{a_4})$
\\ 
\end{tabular}

\mbox{}

\begin{tabular}{|c|c|c|c|c|c|}
\hline
$\delta$ & $J_{*}=(j_{1},j_{2},j_{3},j_{4})$ & $w$ & $N$ & $T$ & deg \\ 
\cline{1-6}
4 & 0000 & 1 & 1 & $s^{2}$ & 6 \\ 
3 & 1000 & $x_{1}$ & 4 & $-s$ & 4 \\ 
2 & 1100 & $x_{1}x_{2}$ & 4 & 1 & 2 \\ 
2 & 1010 & $x_{1}^{2}$ & 2 & $s^{6}$ & \textbf{8} \\ 
1 & 1110 & $x_{1}^{2}x_{2}$ & 4 & $-s^{5}$ & 6 \\ 
0 & 1111 & $x_{1}^{3}x_{2}$ & 1 & $-s^{8}$ & \textbf{8} \\ \hline
\end{tabular}

\mbox{}

\mbox{}\newpage

\begin{tabular}{l}
\textbf{Table\,3}. $L=3,\, V_{J_*}=
V(x_1^{a_1}x_2^{a_2}x_1^{a_3}x_2^{a_4}x_1^{a_5}x_2^{a_6}) $ \\ 
\end{tabular}

\mbox{}

\begin{tabular}{|c|c|c|c|c|c|}
\hline
$\delta$ & $J_{*}=(j_{1},j_{2},j_{3},j_{4},j_{5},j_{6})$ & $w$ & $N$ & $T$ & 
deg \\ \cline{1-6}
6 & 000000 & 1 & 1 & $s^{2}$ & 8 \\ 
5 & 100000 & $x_{1}$ & 6 & $-s$ & 6 \\ 
4 & 110000 & $x_{1}x_{2}$ & 9 & 1 & 4 \\ 
4 & 101000 & $x_{1}^{2}$ & 6 & $s^{6}$ & 10 \\ 
3 & 111000 & $x_{1}^{2}x_{2}$ & 18 & $-s^{5}$ & 8 \\ 
3 & 101010 & $x_{1}^{3}$ & 2 & $s^{9}$ & \textbf{12} \\ 
2 & 111100 & $x_{1}^{3}x_{2}$ & 12 & $-s^{8}$ & 10 \\ 
2 & 110110 & $x_{1}^{2}x_{2}^{2}$ & 3 & $s^{10}$ & \textbf{12} \\ 
1 & 111110 & $x_{1}^{4}x_{2}$ & 6 & $-s^{11}$ & \textbf{12} \\ 
0 & 111111 & $x_{1}^{2}x_{2}x_{1}^{2}x_{2}$ & 1 & $2s^{12}$ & \textbf{12} \\ 
\hline
\end{tabular}

\mbox{}

\mbox{}

\begin{tabular}{l}
\textbf{Table\,4}. $L=4,\, V_{J_*}=
V(x_1^{a_1}x_2^{a_2}x_1^{a_3}x_2^{a_4}x_1^{a_5}x_2^{a_6}x_1^{a_7}x_2^{a_8}) $
\\ 
\end{tabular}

\mbox{}

\begin{tabular}{|c|c|c|c|c|c|}
\hline
$\delta$ & $J_{*}=(j_{1},j_{2},j_{3},j_{4},j_{5},j_{6},j_{7},j_{8})$ & $w$ & 
$N$ & $T$ & deg \\ \cline{1-6}
8 & 00000000 & 1 & 1 & $s^{2}$ & 10 \\ 
7 & 10000000 & $x_{1}$ & 8 & $-s$ & 8 \\ 
6 & 11000000 & $x_{1}x_{2}$ & 16 & 1 & 6 \\ 
6 & 10100000 & $x_{1}^{2}$ & 12 & $s^{6}$ & 12 \\ 
5 & 11100000 & $x_{1}^{2}x_{2}$ & 48 & $-s^{5}$ & 10 \\ 
5 & 10101000 & $x_{1}^{3}$ & 8 & $s^{9}$ & 14 \\ 
4 & 11110000 & $x_{1}^{3}x_{2}$ & 55 & $-s^{8}$ & 12 \\ 
4 & 11011000 & $x_{1}^{2}x_{2}^{2}$ & 13 & $s^{10}$ & 14 \\ 
4 & 10101010 & $x_{1}^{4}$ & 2 & $s^{12}$ & \textbf{16} \\ 
3 & 11111000 & $x_{1}^{4}x_{2}$ & 48 & $-s^{11}$ & 14 \\ 
3 & 10101101 & $x_{1}^{3}x_{2}^{2}$ & 8 & $s^{13}$ & \textbf{16} \\ 
2 & 11111100 & $x_{1}^{2}x_{2}x_{1}^{2}x_{2}$ & 12 & $2s^{12}$ & 14 \\ 
2 & 11110110 & $x_{1}^{5}x_{2}$ & 16 & $-s^{14}$ & \textbf{16} \\ 
1 & 11111110 & $x_{1}^{3}x_{2}x_{1}^{2}x_{2}$ & 8 & $s^{15}$ & \textbf{16}
\\ 
0 & 11111111 & $x_{1}^{3}x_{2}x_{1}^{3}x_{2}$ & 1 & $-s^{16}$ & \textbf{16}
\\ \hline
\end{tabular}

\mbox{}

\mbox{}

From these tables, one can compute the degree and the leading term of $%
P_{J_{*}}^{A_{*}}V_{J_{*}}$ using the formula 
\begin{equation*}
\delta+\deg (T)+\sum_{i=1}^{2L}(3a_{i}-1)=\deg +3D-2L.
\end{equation*}
Next, using the top degree from the table with the corresponding coefficient
and subtracting the degree of the denominator $(s^{2}+1)^{2L}$ we get the
result. For instance, using the Table 4, the leading coefficient of $%
V_3(x_1^{a_1}x_2^{a_2}x_1^{a_3}x_2^{a_4}x_1^{a_5}x_2^{a_6}x_1^{a_7}x_2^{a_8}) 
$, $a_i\geq 2$, is 
\begin{equation*}
2+8-16+8-1=1
\end{equation*}
and its degree is 
\begin{equation*}
16+3D-2\cdot4-2\cdot8=3D-8.
\end{equation*}
\end{proof}

As we can see from the tables, there are braids $\beta ^{A_{\ast }}$ with
positive exponents having maximal degree $\deg V_{3}(\beta ^{A_{\ast
}})=3D-2L$, but not all their exponents are $\geq 2$; other examples with
maximal degree are given by the families $V_{3}(a_{1},1,\dots ,1)$ and $%
V_{3}(a_{1},a_{2},1,\dots ,1)$, $a_{1},a_{2}\geq 2$, $2L$ indices, with
leading terms $s^{3D-2L}$ if $L\equiv 0(\func{mod}3)$ and $-s^{3D-2L}$ if $%
L\equiv 1(\func{mod}3)$. These examples shows that combinatorics of $J_{\ast
}$-blocks of maximal degree in the expansion formula is not so simple.

\begin{conjecture}
\label{con4.15} If $a_{i}\geq2$, the leading term of $%
V_{3}(x_1^{a_1}x_2^{a_2}\dots x_1^{a_{2L-1}}x_2^{a_{2L}})$ is $s^{3D-2L}$.
\end{conjecture}

\section{There are few unit polynomials in a row}

\label{sec5} Now we apply the recurrence relation to evaluate the number of
solutions of the equation $V(e)=1$, where $V(e)=V_{n}(x_{i_{1}}^{a_{1}}\dots
x_{i_{j}}^{e}\dots x_{i_{k}}^{a_{k}})$, with the same conventions: the
sequence $x_{i_{1}},\dots,x_{i_{k}}$ of generators of $\mathcal{B}_{n}$ is
fixed, the exponents $a_{1},\dots,\widehat{a_{j}},\dots a_{k}$ are fixed,
and $e$ is an arbitrary integer.

\begin{proposition}
\label{prop5.1} a) If $V(e)$ and $V(e+1)$ are polynomials in $s$ and $k\geq2$%
, then $V(e+k)$ is a polynomial different from $1$.

b) If $V(e)$ and $V(e-1)$ are polynomials in $s^{-1}$ and $k\geq2$, then $%
V(e-k)$ is a polynomial in $s^{-1}$ different from $1$.
\end{proposition}

\begin{proof}
a) From the recurrence relation $V(e+k)$ is a polynomial in $s$ for $k\geq2$%
. If $V(e+k)=1$ for some $k\geq2$, then $1=(s^{3}-s)V(e+k-1)+s^{4}V(e+k-2)$,
and this is impossible with polynomials. Similarly for part b).
\end{proof}

An obvious consequence of Propositions~\ref{prop3.5} and \ref{prop3.6} is
the fact that in any sequence $\big(V(e)\big)_{e\in \mathbb{Z}}$ there are
only finitely many terms equal to $1$. Now we prove a stronger statement: in
such a sequence there are at most two polynomials equal to $1$, and in this
case they are very close:

\begin{lemma}
\label{lem5.3} The sequence $(V(e))_{e\in \mathbb{Z}}$ could contain at most
two terms equal to $1$. If $V(a)=V(b)=1$, $a\neq b$, then $|a-b|=1\,\mbox{or}%
\,\,2$.
\end{lemma}

\begin{proof}
Suppose $V(a)=1$ and $V(b)=1$ ($a<b$). We can take $Q_{n}=V(n+a)$, where $%
n\in \mathbb{Z}$. Hence $Q_{0}=1$. From the recurrence relation we have 
\begin{equation*}
Q_{n}=\frac{1}{s^{3}+s}[(Q_{1}+s)s^{3n}+(s^{3}-Q_{1})(-s)^{n}].
\end{equation*}
Let us suppose that $Q_{n}=1$ for a positive $n$; in our case $b-a$ is such
an example. If $n$ is even, we have $%
s^{3}+s=(s^{3n}-s^{n})Q_{1}+s^{3n+1}+s^{n+3}.$ This implies $Q_{1}=\frac{1}{%
s^{n}(s^{2n}-1)}[s^{3}+s-s^{3n+1}-s^{n+3}]$ $=-s-\frac{1}{s^{n-1}(s^{n}+1)}%
(s^{2}+1).$ This gives Laurent polynomials only for $n=0$ and $n=2$. In the
last case, $Q_{1}=-s-\frac{1}{s}$, hence only $Q_{0}$ and $Q_{2}$ are equal
to $1$. If $n$ is odd, we have $s^{3}+s=(s^{3n}+s^{n})Q_{1}+s^{3n+1}-s^{n+3}$%
. This implies $Q_{1}=\frac{1}{s^{n}(s^{2n}+1)}[s^{3}+s-s^{3n+1}+s^{n+3}]$ $%
=-s+\frac{1}{s^{n}(s^{2n}+1)}s(s^{2}+1)(s^{n}+1)$. Laurent polynomials are
obtained only for $n=0$ and $n=1$. In this last case, $Q_{0}=Q_{1}=1$ and $%
Q_{2}\ne 1$.
\end{proof}

\begin{example}
\label{ex5.4} The closures of the braids $x_{1}^{-1}$, $x_{1}^{0}$, and $%
x_{1}^{1}$ in $\mathcal{B}_{2}$ give the links in the following figure: 

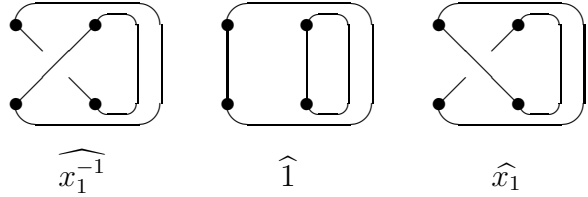
\begin{picture}(360,100)
\put(73,43){\line(1,1){30}}   \put(73,73){\line(1,-1){10}} 
\put(263,43){\line(-1,1){30}} \put(103,43){\line(-1,1){10}} 
\put(130,40){\line(1,1){7}} \put(150,60){\line(-1,-1){7}} 
\multiput(70,40)(80,0){3}{$\bullet$} \multiput(100,40)(80,0){3}{$\bullet$}
\multiput(70,70)(80,0){3}{$\bullet$} \multiput(100,70)(80,0){3}{$\bullet$}
\multiput(153,43)(30,0){2}{\line(0,1){30}} \put(89,12){$\widehat{x_1^{-1}}$}  \multiput(111,43)(80,0){3}{\oval(15,8)[b]} \put(169,12){ $\widehat{1}$ } \multiput(111,73)(80,0){3}{\oval(15,8)[t]} \put(249,12){ $\widehat{x_1}$ } 
\multiput(100,43)(80,0){3}{\oval(55,16)[b]} \put(233,43){\line(1,1){10}} 
\multiput(100,73)(80,0){3}{\oval(55,16)[t]} \put(263,73){\line(-1,-1){10}} 
\multiput(128,43)(80,0){3}{\line(0,1){30}} 
\multiput(119,43)(80,0){3}{\line(0,1){30}} 
\end{picture}

\noindent and the corresponding Jones polynomials are: $V_{2}(-1)=1$, $%
V_{2}(0)=-s-s^{-1}$, and $V_{2}(1)=1$. This shows that the case $%
V(a)=V(a+2)=1$ is possible.
\end{example}

\begin{lemma}
\label{lem5.5} Let $\beta(e)=\alpha x_{i}^{e}\gamma$ and $\beta(e+1)=\alpha
x_{i}^{e+1}\gamma$ be two braids in $\mathcal{B}_{n}$. Then $\widehat{\beta
(e)}$ and $\widehat{\beta (e+1)}$ cannot be simultaneously knots. In
particular, $V(e)$ and $V(e+1)$ evaluated at $1$ cannot be $1$ at the same
time.
\end{lemma}

\begin{proof}
If $\widehat{\beta (e)}$ is a knot, then the associated permutation $\pi
(\beta (e))$ is an $n$-cycle. The permutation associated with $\beta (e+1)$
is $\pi (\alpha x_{i}^{e+1}\gamma )=\pi (\alpha x_{i}^{e}\gamma )\pi (\gamma
^{-1}x_{i}\gamma )$, and this cannot be an $n$-cycle because the signature
of the last factor is $-1$.
\end{proof}

\emph{Proof of Theorem}~\ref{thm1.7} This is a consequence of lemmas~\ref%
{lem5.3} and \ref{lem5.5}.


\section{APPENDIX}

\label{sec6} In~\cite{nizami:08} multiple Fibonacci sequences (and multiple
Fibonacci modules) are introduced. A multiple sequence $(x_{n_{1},\ldots
,n_{p}})_{n_{1},\dots ,n_{p}\in \mathbb{Z}}$ of elements in a ring $\mathcal{%
R}$ is called a \emph{multiple Fibonacci sequence} of type $(\beta ,\gamma
)\in \mathcal{R}^{2}$ if for any $i\in \{1,\ldots ,p\}$ and any $k\in 
\mathbb{Z}$ we have 
\begin{equation*}
a_{n_{1},\dots ,n_{i-1},k+2,n_{i+1},\dots ,n_{p}}=\beta a_{n_{1},\dots
,n_{i-1},k+1,n_{i+1},\dots ,n_{p}}+\gamma a_{n_{1},\dots
,n_{i-1},k,n_{i+1},\dots ,n_{p}}.
\end{equation*}%
The $\mathcal{R}$-module of such multiple sequences is denoted by $\mathcal{F%
}^{[p]}(\beta ,\gamma )$ and it is isomorphic with the tensor product $%
\mathcal{F}(\beta ,\gamma )^{\otimes p}$.

\begin{theorem}
\label{thm6.1} Let $(x_{n_{1},\dots ,n_{p}})_{\geq0}$ be an element in $%
\mathcal{F}^{[p]}(r_{1}+r_{2},-r_{1}r_{2})$.\newline
\textbf{a)} The general term is given by 
\begin{equation*}
x_{n_{1},\dots ,n_{p}}=D^{-p}\sum_{0\leq j_{1},\dots ,j_{p}\leq 1}
S_{j_{1}}^{[n_{1}]}(r_{1},r_{2})\dots
S_{j_{p}}^{[n_{p}]}(r_{1},r_{2})x_{j_{1},\dots ,j_{p}}\,,
\end{equation*}
where $D=r_{2}-r_{1}$, $%
S_{0}^{[n]}(r_{1},r_{2})=r_{1}^{n}r_{2}-r_{1}r_{2}^{n}$, $%
S_{1}^{[n]}(r_{1},r_{2})=r_{2}^{n}-r_{1}^{n}$.\newline
\textbf{b)} The generating function of $(x_{n_{1},\dots ,n_{p}})$ is given
by 
\begin{equation*}
G(t_{1},\dots ,t_{p})=q(t_{1})^{-1}\dots q(t_{p})^{-1}\sum_{0\leq
j_{1},\dots ,j_{p}\leq 1} Q_{j_{1}}(t_{1})\dots
Q_{j_{p}}(t_{p})x_{j_{1},\dots ,j_{p}}\,,
\end{equation*}
where $q(t)=(1-r_{1}t)(1-r_{2}t)$, $Q_{0}(t)=1-(r_{1}+r_{2})t$ and $%
Q_{1}(t)=t$.
\end{theorem}

\begin{proof}
These require only elementary computations; one can find all the details in~%
\cite{nizami:08}.
\end{proof}



\end{document}